\documentclass[11pt]{article}
\usepackage{amsthm,mathrsfs,amssymb}
\usepackage{baskervald}
\usepackage{stmaryrd}
\usepackage{amsmath,verbatim}
\usepackage{hyperref}
\usepackage[numbers]{natbib}
\usepackage[active]{srcltx}
\usepackage[matrix,arrow,curve]{xy}
\usepackage{bbm}
\usepackage{verbatim}
\usepackage{setspace}
\usepackage{parskip}
\usepackage[margin=2.5cm]{geometry}
\usepackage{baskervald}

\newtheoremstyle{mytheoremstyle} 
{\topsep}                    
{\topsep}                    
{\itshape}                   
{1.5em}                           
{\itshape\bf}                   
{.---}                          
{.50em}                       
{}  
\theoremstyle{mytheoremstyle}

\newtheorem{para}{}[subsection]
\newtheorem{thm}[para]{Theorem}
\newtheorem{prop}[para]{Proposition}
\newtheorem{lemma}[para]{Lemma}
\newtheorem{cor}[para]{Corollary}

\theoremstyle{definition}
\newtheorem{defn}{Definition}[section]

\numberwithin{equation}{section}
	\usepackage{verbatim}

\newcommand{\calE}{\mathcal{E}}
\newcommand{\calF}{\mathcal{F}}

\newcommand{\calH}{\mathcal{H}}

\newcommand{\calL}{\mathcal{L}}

\newcommand{\calO}{\mathcal{O}}

\newcommand{\calR}{\mathcal{R}}

\newcommand{\calU}{\mathcal{U}}
\newcommand{\calV}{\mathcal{V}}

\newcommand{\fraka}{\mathfrak{a}}

\newcommand{\frakt}{\mathfrak{t}}

\newcommand{\frakL}{\mathfrak{L}}

\newcommand{\HH}{\mathbb{H}}

\newcommand{\NN}{\mathbb{N}}

\newcommand{\PP}{\mathbb{P}}
\newcommand{\QQ}{\mathbb{Q}}
\newcommand{\RR}{\mathbb{R}}

\newcommand{\ZZ}{\mathbb{Z}}

\renewcommand{\ker}{\operatorname{ker}}

\newcommand{\so}{\Longrightarrow}

\newcommand{\rar}{\rightarrow}

\newcommand{\wh}{\widehat}
\newcommand{\ul}{\underline}
\newcommand{\ol}{\overline}

\newcommand{\SO}{\operatorname{SO}}

\newcommand{\Span}{\mathrm{Span}}

\newcommand{\Lat}{\frakL \fraka \frakt}

\renewcommand{\ker}{\operatorname{ker}}
\renewcommand{\Re}{\operatorname{Re}}

\newcommand{\GL}{{\mathrm{GL}}}

\newcommand{\diag}{\mathrm{diag}}

\renewcommand{\mod}{\text{mod }}

\newcommand{\rN}{\mathrm{N}}

\newcommand{\dbq}[1]{{[\![#1]\!]}}

\newlength{\dhatheight}

\newcommand{\pard}[2]{{\frac{\partial{#1}}{\partial{#2}}}}

\newcommand{\regR}{\mathrm{R}}

\newcommand{\rT}{\mathrm{T}}

\numberwithin{equation}{section}

\title{Period Relations for Quaternionic Elliptic Functions}
\author{Z. Amir-Khosravi}
\begin{document}
\maketitle
\begin{abstract}We study elliptic functions in quaternionic analysis, and prove some analogues of classical theorems from the complex case. The main result is a relation between the periods of closed differential $1$-forms and $3$-forms on $\HH/L$, where $L\subset \HH$ is a lattice. By applying it to forms that are regular in the sense of quaternionic analysis, we obtain quaternionic analogues of Riemann's period relations for an abelian surface. We also obtain a quaternionic analogue of Legendre's relation, as an identity between $2\pi^2$ and a linear combination of the four quasi-periods of the quaternionic Weierstrass $\zeta$-function, where the coefficients are rational functions of a basis for $L$. When $L$ is a left-ideal for a maximal order in a definite quaternion algebra over $\QQ$, we obtain three extra relations among the quasi-periods. 
\end{abstract}

\tableofcontents

\section{Introduction}

Let $X=\HH/L$, where $L\subset \HH$ is a full-rank lattice with generators $\lambda_1,\cdots,\lambda_4$. For $d=1,\cdots,3$ we write $\calE^{(d)}=\calE^{(1)}(X)$ for the set of closed $\HH$-valued differential $d$-forms on $X$, with possible singularities on the lattice $L$. The generators $\lambda_1,\cdots, \lambda_4$ determine homology classes $\gamma_1,\cdots, \gamma_4\in H_1(X,\ZZ)$, with respect to which we define the \textit{period} of a closed $1$-form $\omega\in \calE^{(1)}$ to be the column vector in $\HH^4$ given by
$$ P(\omega) = {}^t \left(\int_{\gamma_1} \omega,\cdots, \int_{\gamma_4}\omega\right).$$
Likewise, $\lambda_1,\cdots,\lambda_4$ determine a basis $\xi^{(1)},\cdots,\xi^{(4)}$ for $H_3(X,\ZZ)$, where $\xi^{(i)}$ is a face of a fundamental parallelogram $\xi$ determined by $\lambda_i$. For each $\omega \in \calE^{(3)}$ we also define a period vector
$$ P(\omega) = {}^t \left(\int_{\xi^{(1)}}\omega ,\cdots, \int_{\xi_{(4)}} \omega\right).$$

For $\omega_1 \in \calE^{(1)}$ and $\omega_3 \in \calE^{(3)}$, we can write $\omega_1\wedge \omega_3 = F(q) dv$, where $dv$ is the volume form, and $F(q)$ a function defined on $U_r = \mathrm{int}(\xi) - B_r(0)$, with $B_r(0)$ the ball of radius $r$ at $0$. Writing $\omega_1 = df$, for $f$ defined on $U_r$, the basic period relation is the following consequence of Stokes's theorem (Proposition \ref{fprop}) :
		$$ {}^t P(\omega_3) J P(\omega_1) = - \int_{U_r} F(q) dv + C_r (\omega_3 f),$$
where $J = \diag(1,-1,1,-1)$, and
$$ C_r (\omega) = \int_{\partial B_r(0)} \omega.$$

To obtain relations among the periods of $1$-forms, we consider the map
$$ \calR:\ \ \calE^{(1)} \rar \calE^{(3)},\ \  \calR(\omega) = \omega \wedge dq \wedge dq.$$
Then we have
\begin{align}\label{P1} P(\calR(\omega)) = P(\omega) J Q,\end{align}
(Lemma \ref{Rlem}) where $Q$ is a $4\times 4$ quaternionic hermitian matrix obtained from $\lambda_1,\cdots,\lambda_4$ (\ref{Q}).

Riemann's period relations are a consequence of similar considerations as above, combined with the fact that a function of one complex variable $z=x+iy$ is holomorphic if and only if $\omega=df$ satisfies $dz \wedge \omega = 0$, in which case 
$$\frac{1}{2i} \ol{\omega} \wedge \omega = |f'(z)|^2 dx \wedge dy.$$
The corresponding facts in quaternionic analysis are as follows: a function $f: U \rar \HH$ on a domain  $U\subset \HH$ is left-regular if and only if $Dq \wedge df = 0$, in which case  (Lemma \ref{dvlem})
$$\calR(\ol{df}) \wedge df = -| \partial_l f(q)|^2 dv.$$
Here $Dq$ is a distinguished closed $3$-form (\ref{Dq}) that plays a role analogous to $dz$, and $\partial_l$ is the left-derivative of $f$, which exists when $f$ is regular. A similar result holds for right-regular functions.

Let $\Lambda\in \GL_4(\RR)$ be the matrix of the real components of $\lambda = {}^t (\lambda_1,\cdots,\lambda_4)$, so that
$$ \lambda = \Lambda\cdot \rho,\ \ \ \rho = {}^t (1,i,j,k),$$
and put
$$ \wh{\lambda}= \det(\Lambda)  \Lambda^{-1}\cdot \rho.$$

We say $\omega = df\in \calE^{(1)}$ is \textit{meromorphic} if $f(q)$ is regular on $\HH-L$, with possible poles on $L$. Our main result contains the following (see Theorem \ref{thm1} for the complete statement.)

\textbf{Theorem} (Quaternionic Period Relations) 
\begin{itemize} \item[(i)]  If $\omega=df$ is meromorphic on $X$, then
	$${}^t \wh{\lambda}\cdot P(\omega) = 2\pi^2 \mathrm{res}_0(f).$$
	\item[(ii)] If $\omega=df$ is regular on $X$,
	$$ {}^t \ol{P(\omega)} Q P(\omega) = -\int_{X} |\partial_l f|^2 dv.$$
	\item[(iii)] If $\omega$ and $\omega'$ are both regular on $X$,
	$$ {}^t P(\omega') Q P(\omega) = 0.$$
\end{itemize}

Here $Q$ is the same matrix (\ref{Q}) that occurs in (\ref{P1}). Parts $(ii)$ and $(iii)$ are the analogues of Riemann's bilinear relations for holomorphic $1$-form on an abelian surface. Together they imply, for instance, that regular $1$-form on $X$ are uniquely determined by their period vectors, and span a right $\HH$-vector space of dimension $3$.

The quaternionic Weierstrass-$\zeta$ function is a regular function on $\HH-L$ that satisfies the quasi-periodicity relations
$$ \eta_i= \zeta(q+\lambda_i) - \zeta(q),\ \ i=1,\cdots,4.$$
As in the complex case, the constants $\eta_i$ are closely related to normalized (quaternion-valued) Eisenstein series of smallest weight associated with $L$. Since the column vector ${}^t\eta = (\eta_1,\cdots, \eta_4)$ is the period vector of the meromorphic form $d\zeta$, from part $(i)$ of the theorem applied to $\omega = d\zeta$ we obtain
\begin{align}\label{Legendre} {}^t \wh{\lambda} \cdot \eta = 2\pi^2.\end{align}
The above is the analogue of the classical relation 
$$ \lambda_2 \eta_1 - \lambda_1 \eta_2 = 2\pi i,$$
due to Legendre, between the periods $\lambda_i$ of the Weierstrass $\wp$-function and quasi-periods $\eta_i$ of the Weierstrass $\zeta$-function.

When $L$ has quaternionic multiplication, we can deduce further relations as follows.

\textbf{Theorem}\textit{
Let $D\subset \HH$ be a division algebra over $\QQ$, and $L\subset D$ a left-ideal for a maximal order  $\calO$. Let $\calO \rar M_4(\ZZ),\ a \mapsto U_a$ be the integral representation of $\calO$ determined by $a\lambda = U_a \lambda$. Then for each $a\in \calO$,
$$ {}^t \wh{\lambda} \cdot U_a\cdot \eta = 2\pi^2 \ol{a} .$$}
It's clear that if $a_1,\cdots, a_4$ form a $\ZZ$-basis for $\calO$, the relations in the theorem are integer combinations of the corresponding four relations, of which the case $a=1$ is (\ref{Legendre}.

\section{Regular Quaternionic Functions}

In this section we briefly review the basic facts about regular quaternionic functions, and establish notation. The standard reference is \cite{Sud79}. 

\subsection*{Notation}
The quaternions $i$, $j$, $k$ will be denoted $e_1$, $e_2$, $e_3$, so that $e_i^2=e_1 e_2 e_3 = -1$. The space of quaternions is denoted $\HH$, and an element $q\in \HH$ is written
$$ q = t + e_1 x_1 + e_2 x_2 + e_3 x_3.$$
Occasionally, it will be useful to write $q = \sum_{i=0}^3 e_i x_i$, so we set $e_0=1$ and $x_0 = t$. The standard involution is denoted $q \mapsto \ol{q}$. The subsets $\HH_0$, $\HH^1 \subset \HH$ denote of quaternions of reduced trace zero and norm one, respectively, which are the functions 
$$ \rT(q) = q + \ol{q},\ \ \ \rN(q) = \ol{q} q.$$

There will be several families of functions $\{F_\nu\}$ indexed by $\nu=(n_1,n_2,n_3) \in \NN^3$. We put
$$ |\nu| = n_1 + n_2 + n_3.$$
We write $\nu\geq 0$, resp. $\nu>0$, if $|\nu|\geq 0$, resp. $|\nu|>0$. For $\nu_1,\nu_2\in \NN^3$, the relations $\nu_1 \geq\nu _2$ and $\nu_1 >\nu_2$ are defined in the obvious way.
 For $n\in \NN$, we put
\begin{align}\sigma_n = \{\nu=(n_1,n_2,n_3)\in \NN^3: |\nu|=n\}.
\end{align}
For Greek letters $\mu$, $\nu$, $\lambda$, \textit{etc.} in $\NN^3$, the corresponding Latin minuscule $m$, $n$, $l$, \textit{etc.} will often denote $|\mu|$, $|\nu|$, $|\lambda|$, \textit{etc.} as the context will make clear.

We will make use of the following sum notations:
$$ \sum_{\nu\geq 0} F_{\nu} = \sum_{n=0}^\infty \sum_{\nu\in \sigma_n} F_{\nu},\ \ \ \ \sum_{\nu>0} F_{\nu}= \sum_{n=1}^\infty \sum_{\nu\in \sigma_n} F_{\nu}.$$
We will also write $F_{(1)}$, $F_{(2)}$, $F_{(3)}$, or $F_{100}$, $F_{010}$, $F_{001}$, instead of $F_{(1,0,0)}$, $F_{(0,1,0)}$, $F_{(0,0,1)}$, respectively.

\subsection{Regular Homogeneous Functions}

If $f(q)$ is an $\HH$-valued differentiable function on $\HH$, the following differential operators act on it:
\begin{align}
\begin{split}
\partial_l f &= \frac{1}{2} \left( \pard{f}{t} - e_1 \pard{f}{x_1} - e_2 \pard{f}{x_2} - e_3 \pard{f}{x_3}\right),\\
\ol{\partial}_l f &= \frac{1}{2} \left( \pard{f}{t} + e_1 \pard{f}{x_1} + e_2 \pard{f}{x_2} + e_3 \pard{f}{x_3}\right),\\
\partial_r f &= \frac{1}{2} \left( \pard{f}{t} - \pard{f}{x_1} e_1 - \pard{f}{x_2} e_2 - \pard{f}{x_3} e_3\right),\\
\ol{\partial}_l f &= \frac{1}{2} \left( \pard{f}{t} + \pard{f}{x_1}e_1 + \pard{f}{x_2} e_2 +  \pard{f}{x_3}e_3\right).
\end{split}
\end{align}
The function $f$ is \textit{left-regular} if $\ol{\partial}_l f = 0$, and \textit{right-regular} if $\ol{\partial}_r f = 0$. By a \textit{regular} function we always mean a \textit{left-regular} one. The set of left-regular (resp. right-regular) functions form a right (resp. left) $\HH$-vector space.

A function $f:\HH^\times \rar \HH$ is called \textit{homogeneous of degree} $n\in \ZZ$ if $f(rq)=r^n f(q)$ for all $q\in \HH^\times$ and $r\in \RR^\times$. The right vector space of left-regular homogeneous functions of degree $n$ is denoted $U_n$. 

The functions 
\begin{align} z_i = t e_i - x_i,\ \ \ i=1,2,3\end{align}
are both left and right regular, and form an $\HH$-basis for $U_1$. 

For $\nu=(n_1,n_2,n_3)\in \NN^3$ and $n=|\nu|$, let $I_{\nu}$ denote the set of $n$-tuples $(i_1,\cdots, i_n)$ such that
$$\{i_1,\cdots, i_n\}=\{n_1\cdot 1, n_2\cdot 2, n_3\cdot 3\}$$ as multi-sets. The distinguished regular homogeneous polynomials $P_{\nu}$ are defined as
\begin{align} \label{Pnu} P_{\nu}(q) = \frac{1}{n!} \sum_{(i_1,\cdots, i_n)\in I_{\nu}} 
z_{i_1} z_{i_2}\cdots z_{i_n}.
\end{align}
For the case $n=0$, we put 
$$P_{000}(q)=1.$$
Then for all $n\geq 0$, the set $\{ P_{\nu}(q): \nu \in \sigma_n\}$ is a right $\HH$-basis for $U_n$. In particular, 
$$ \dim_{\HH} U_n = \frac{1}{2}(n+1)(n+2).$$

In negative degrees we have $U_{-1}=U_{-2}=\{0\}$. The space $U_{-3}$ has dimension $1$ over $\HH$, and is spanned by 
\begin{align}\label{G(q)} G(q) = \frac{q^{-1}}{\rN q}.\end{align}
For each $\nu\in \NN^3$, we define a differential operator 
\begin{align}\partial_\nu = \frac{\partial^n}{\partial x_1^{n_1} \partial x_2^{n_2} \partial x_3^{n_3}}.
\end{align}
The functions
\begin{align} G_{\nu}(q) = \partial_{\nu} G(q)\end{align}
are left and right regular, and homogeneous of degree $-n-3$. The set 
$\{G_{\nu}(q):\ |\nu|=n.\}$
is an $\HH$-basis for $U_{-n-3}$. Clearly,
$$ \partial_{\mu} G_{\nu} = G_{\nu+\mu}.$$
It follows from $\frac{d}{dx_i}z_i=-1$ that 
\begin{align} \label{dP} \partial_\nu 
P_\mu=\left\{\begin{array}{cc} (-1)^n P_{\mu-\nu}&\text{ if }\mu\geq \nu, 
\\0 &\text{ otherwise.}\end{array}\right.
\end{align}

There is a distinguished $\HH$-valued $1$-form 
\begin{align} dq = dt + e_1 dx_1 + e_2 dx_2 + e_3 dx_3\end{align}
and a distinguished $3$-form given by
\begin{align}\label{Dq} Dq = dx_1 \wedge dx_2 \wedge dx_3 - e_1 dt \wedge dx_2 \wedge dx_3 + e_2 dt \wedge dx_1 \wedge dx_3 - e_3 dt \wedge dx_1 \wedge dx_3.\end{align}

If $U\subset \HH$ is a domain, and $f: U \rar \HH$ is differentiable, it determines a $1$-form
$$ df = \pard{f}{t} dt + \pard{f}{x_1}dx_1 + \pard{f}{x_2}dx_2 + \pard{f}{x_3} dx_3.$$

The function $f$ is regular if and only if $Dq \wedge df = 0$. If $f$ and $g$ are left and right-regular functions on $U$, respectively, the form $g Dq f$ is closed.

Let $B_r(a)$ denote the ball of radius $r$ at $a$. If $f$ is regular on $B_r(a) -\{a\}$, it has a unique series expansion of the form
\begin{align}\label{fseries} f(q) = \sum_{\nu \geq 0} P_{\nu}(q-a)a_{\nu}  + G_{\nu}(q-a)b_{\nu}\end{align}
where 
\begin{align}\label{fcoeffs} a_{\nu} = \frac{1}{2\pi^2} \int_{S} G_{\nu} (q-a) Dq f(q),\ \ \ b_{\nu} = \frac{1}{2\pi^2} \int_{S} P_{\nu}(q-a) Dq f(q),\end{align}
for $S$ any differentiable $3$-chain homologous to $\partial B_{r}(a)$ in the domain of $f$.

Clearly $f: B_r(a) \rar \HH$ can be extended to a regular function at $a$ if and only if all $b_{\nu}=0$. In that case
$$ a_{\nu} = (-1)^{n} \partial_{\nu} f (a)$$
so that
\begin{align}\label{fexp} f(q-a) = \sum_{\nu \geq 0} P_{\nu}(q) \partial_{\nu} f (a) .\end{align}

We say $f: U-\{a\} \rar \HH$ is \textit{meromorphic} at $a$ if only finite many $b_{\nu}$ in (\ref{fseries}) are non-zero. In that case the \textit{residue} of $f$ at $a$ is defined and denoted by $$\mathrm{res}_{a} (f)=b_{0} = \frac{1}{2\pi^2} \int_S Dq f(q).$$

The group $\HH^\times \times \HH^\times$ acts on the space of regular functions $f: \HH^\times \rar\HH$ by
$$ u=(a,b)\in \HH^\times\times\HH^\times,\ \ \ (u\cdot f)(q) = b f(a^{-1}qb).$$
Each $U_n$ becomes a representation of $\HH^\times \times \HH^\times$ in this way.

If $f:U \rar \HH$ is regular at $a\in U$, and $a\neq 0$, the function
\begin{align} (\regR f)(q) = \frac{q^{-1}f(q^{-1})}{\rN q}\end{align}
is regular at $a^{-1}$. The map $f \mapsto \regR (f)$ is an involution, and induces isomorphisms of representations $U_n \rar U_{-n-3}$ for each $n\in\ZZ$. 

\subsection{Supplements on Polynomials and Power Series}

Some of the material in this subsection is implicit in the literature.

 The following proposition is the quaternionic analogue of the 
basic identity
$$\frac{z^m}{m!}\frac{z^{n-m}}{(n-m)!}  = \binom{n}{m} 
\frac{z^{n}}{n!}.$$

\begin{prop}\label{binom} For integers $n>m>0$ and $\nu\in \sigma_n$,
	$$ \sum_{\tiny \begin{array}{c}|\mu|=m\\\mu<\nu\end{array}}P_{\mu} 
	P_{\nu-\mu} = \binom{n}{m} P_{\nu}$$
\end{prop}
\begin{proof}This follows from the definition (\ref{Pnu}) by the following combinatorial argument. One identifies the indexing set of the monomials occuring in $P_{\nu}(q)$ with the set of paths contained in $[0,n]^3$, from $0$ to $\nu$, in 
	positive unit-steps along the cardinal directions. Multiplication of the monomials in $z_i$ then corresponds to composition of paths. The identity reflects the fact that for $m<n$ fixed, each path 
	from 
	$0$ to $\nu$ is uniquely the composition of a path of length $m$ from 
	$0$ to some $\mu< \nu$, followed with another of length $n-m$ from $\mu$ to $\nu$.
\end{proof}

\begin{prop}\label{Gexp}For $|q|<|p|$, and $\mu\in \NN^3$,
	\begin{align*} (a)\ \ G_{\mu}(p-q) = &\sum_{\nu\geq 0} P_\nu(q) G_{\nu+\mu}(p) 
	= \sum_{\nu\geq 0}	G_{\nu+\mu}(p) P_\nu(q). \\
	(b)\ \ G_{\mu}(p+q) = &\sum_{\nu \geq 0} (-1)^{n} P_\nu(q) G_{\nu+\mu}(p) 
	= \sum_{\nu\geq 0} (-1)^{n}
	G_{\nu+\mu}(p) P_\nu(q). 
	\end{align*}
	The convergence is uniform on $\{(p,q): |q| \leq r|p|\}\subset \HH^2$, 
	for 
	any $r<1$.
\end{prop}
\begin{proof}Clearly $(a)$ and $(b)$ become equivalent via $q\rar -q$. The case $\mu=0$ of $(a)$ is
	$$ G(p-q) = \sum_{\nu\geq 0} P_\nu(q) G_{\nu}(p) 
	= 	\sum_{\nu\geq 0}
	G_{\nu}(p) P_\nu(q),$$
	which is proved in \cite[Prop. 10]{Sud79}. The case $\mu=0$ of $(b)$ then 
	follows:
	$$ G(p+q) = \sum_{\nu \geq 0} (-1)^{n} P_{\nu}(q) 
	G_{\nu}(p) = \sum_{\nu\geq 0} (-1)^{n}G_{\nu}(p)P_{\nu}(q) 
	.$$
	Applying $\partial_\mu$, by uniform convergence and the relation (\ref{dP}), we 
	have
	\begin{align*} G_{\mu} (p+q) &= \partial_{\mu} G(p+q) = \sum_{\nu \geq 0} 
	(-1)^n (\partial_{\mu}P_{\nu}(q)) G_{\nu}(p) = \sum_{n=m}^\infty \sum_{{\tiny 
			\begin{array}{c}|\nu|=n\\\nu\geq\mu\end{array}}}
	(-1)^{n-m} P_{\nu-\mu}(q) G_{\nu}(p).
	\end{align*}
	After the change of variable $\rho=\nu-\mu$, we obtain
	$$G_\mu(p+q)=\sum_{\xi\geq 0} (-1)^{x} P_{\xi}(q) 
	G_{\mu+\xi}(p),$$
	which is $(b)$, from which $(a)$ follows.
\end{proof}

For each $\nu\in \NN^3$, we have
\begin{align}\label{bound}|P_{\nu}(q)|\leq \frac{1}{n!} \sum_{\ul{i}\in S_\nu} |z_{i_1}| 
|z_{i_2}| \cdots |z_{i_n}| = \frac{|z_1|^{n_1}}{n_1!} \frac{|z_2|^{n_2}}{n_2!} \frac{|z_3|^{n_3}}{n_3!}.\end{align}
If $p=-e_1 x_1 - e_2 x_2 - e_3 x_3$ for $x_i>0$, then $z_i(p) = x_i = |z_i(p)|$, and the upper bound is attained.

Since $\{\regR P_{\nu}(q): |\nu|=n\}$ and $\{G_{\nu}(q): |\nu|=n\}$ are two bases for 
the same right $\HH$-vector space, for each $n$ there exists some constant 
$C_n>0$ such that
\begin{align}\label{Gest} |G_{\nu}(q)|\leq \frac{C_n}{|q|^{n+3}}.\end{align}

Let $Z_1$, $Z_2$, $Z_3$ denote abstract variables, and for $\nu\in \sigma_n$ write
\begin{align}Z^{\nu} = Z_1^{n_1} Z_2^{n_2} Z_3^{n_3}\in \RR[Z_1,Z_2,Z_3].
\end{align}
The map $P_{\nu}(q) \rar Z^{\nu}$ is then a bijection between the regular homogeneous polynomials $P_{\nu}(q)$ and monic monomials in $\RR[Z_1,Z_2,Z_3]$. We have
$$ P_{(j)} (e_1 x_1 + e_2 x_2 + e_3 x_3) = - x_j,\ \ j=1,2,3,$$
therefore 
$$ P_{\nu}(e_1 x_1 + e_2 x_2 + e_3 x_3) = (-1)^n x_1^{n_1} x_2^{n_2} x_3^{n_3}.$$

We define the (non-commutative) ring of quaternionic polynomials in $Z_1$, $Z_2$, $Z_3$ to be
\begin{align} \HH[Z_1,Z_2,Z_3] = \RR[Z_1,Z_2,Z_3] \otimes_{\RR} \HH,\end{align}
and the ring of formal quaternionic power series by 
\begin{align} \HH\dbq{Z_1,Z_2,Z_3} = \RR\dbq{Z_1,Z_2,Z_3}\otimes_{\RR} \HH.\end{align}
If $f: U \rar \HH$ is a regular function in a neighbourhood of $0$, the series expansion
$$ f(q) = \sum_{\nu\geq 0} P_{\nu}(q) a_n$$
can be formally encoded as the element of $\HH[Z_1,Z_2,Z_3]$ given by
$$ F(Z_1,Z_2,Z_3) = \sum_{\nu\geq 0} \frac{Z_1^{n_1}}{n!} \frac{Z_2^{n_2}}{n_2!} \frac{Z_3^{n_3}}{n_3!} (-1)^n a_{\nu} .$$
It follows from (\ref{bound}) that $F$ and the series for $f$ have the same radius of convergence. The function $f \mapsto F$ then maps $U_n$ isomorphically onto the subspace of homogeneous polynomials of degree $n$ in $\HH[Z_1,Z_2,Z_3]$.

Let $U_0 = U \cap \HH_0$. For any $p = e_1x_1 + e_2 x_2 + e_3 x_3\in U_0$, we have
$$ f(p) = F(x_1,x_2,x_3).$$
In particular, the series $F$ is absolutely convergent on $\{(x_1,x_2,x_3): \sum_{i=1}^3 e_i x_i\in U_0\}$, and defines a real-analytic function $U_0 \rar \HH$. 

The following definition and proposition are from Clifford analysis. See for instance \cite[14.3, 14.4]{BDS82}.

\begin{defn}Suppose that $U_0\subset \HH_0$ and $U\subset \HH$ are open sets, such that $U_0 \subset U$. Then $U$ is called an \textit{$x_0$-normal} open neighbourhood of $U_0$ if for every $u\in U$, the set $\{t+u: t\in \RR\}\cap U$ is connected, and contains exactly one point of $U_0$.
\end{defn}

\begin{prop}Suppose $U_0\subset \HH_0$ is a domain, and $f_0: U_0 \rar \HH$ a real-analytic function. There exists a maximal $x_0$-normal neighbourhood $U$ of $U_0$, and a unique regular function $f: U\rar \HH$ such that $f_0 = f|_{U_0}.$
\end{prop}

Let $\calF$ be the sheaf of regular functions on $\HH$, and 
$$ \calF_a = \varinjlim_{U\ni a} \calF(U)$$
the stalk of $\calF$ at $a\in \HH$. The functions
$$ s_{r,a}: \calF(B_r(a)) \rar \HH\dbq{Z_1,Z_2,Z_3},\ \ \ \sum_{\nu\geq 0} P_{\nu}(q-a) a_{\nu}  \mapsto \sum_{\nu \geq 0} (-1)^n Z^{\nu} a_{\nu}$$
are injective, and induce an embedding
$$ s_a : \calF_a \rar \HH\dbq{Z_1,Z_2,Z_3},$$
whose image coincides with absolutely convergent power series in $\HH\dbq{Z_1,Z_2,Z_3}$.

Let $U\subset \HH$ be an $x_0$-normal neighbourhood of $U_0 \subset \HH_0$. If $f$ and $g$ are regular functions defined on $U$, then $f_0 g_0: U_0 \rar \HH$ is a real-analytic function on $U_0$, which by the proposition above is necessarily of the form $h|_{U_0}$ for a regular function $h: U' \rar \HH$ on another $x_0$-normal neighbourhood of $U_0$. The function $h$ is called the \textit{CK-product} (Cauchy-Kovalevskaya) of $f$ and $g$, denoted $f* g$. If $a\in U_0$, and $F$, $G$ are the series expansions of $f$ and $g$ at $a$, then $FG$ is the series expansion of $h$. It follows that 
$$ s_a ( f * g) = s_a(f) s_a(g).$$
Thus the ring structure induced by the CK-product on the stalks $\calF_a$ corresponds with the power series ring structure on the series $\HH[Z_1,Z_1,Z_2]$. 

Let $B_r(a)^\times = B_r(a) - \{a\}$ and 
$$ \calF_a^{\times} = \varinjlim_{r>0} \calF(B_r(a)^\times).$$
Then $\calF_a^{\times}$ denotes the germ of regular functions in a punctured neighbourhood of $a$. Clearly there's an injection $\calF_a \rar \calF_a^\times$ induced by restrictions $\calF(B_r(a)) \rar \calF(B_r(a)^\times)$. The CK-product again equips $\calF_a^\times$ with a ring structure, and the map $\calF_a^\times \rar \calF_a$ is a ring homomorphism.

Let
$$ N = Z_1^2 + Z_2^2 + Z_3^2\in \HH\dbq{Z_1,Z_2,Z_3},$$
and
$$\HH\dbq{Z_1,Z_2,Z_3,N^{-1}} = \RR\dbq{Z_1,Z_2,Z_3,N^{-1}}\otimes_\RR \HH.$$

\begin{prop}There's a unique continuous ring homomorphism $s_a^\times: \calF_a^\times \rar  \HH\dbq{Z_1,Z_2,Z_3,N^{-1}}$ such that the following diagram commutes:
	$$\xymatrix{\calF_a \ar[d] \ar[r]^-{s_a} & \HH\dbq{Z_1,Z_2,Z_3} \ar[d] \\
		\calF_a^\times \ar[r]_-{s_a^\times} & \HH\dbq{Z_1,Z_2,Z_3,N^{-1}}.}$$
\end{prop}
\begin{proof}We claim that for each $f\in \calF(B_r(a)^\times)$ there's a unique series $F\in\HH\dbq{Z_1,Z_2,Z_3,N^{-1}}$ that converges on $B_r(a)^\times \cap \HH_0$, and coincides with the value of $f$ there. Then $s_a^{\times}(f)=F$ satisfies all required properties, and is unique because it's determined by its values.
	
Since $f$ is determined by the series expansion (\ref{fseries}), it's enough to show the claim for the functions $G_{\nu}(q)$. In other words, that the restriction of $G_{\nu}(q)$ to $U_0=\HH_0\cap U$ coincides with the value of a rational function of the form $N^{-k} H$, where $H\in \HH\dbq{Z_1,Z_2,Z_3}$. For $\nu=0$ one can take $k=1$ and $H_{0} = - Z_1 e_1 - Z_2 e_2 -Z_3 e_3$. Now applying $\partial_{\nu}$ to a function $f$ corresponds to applying $\pard{}{Z^{\nu}}$ to the rational function $F$ representing it on $U_0$. Since the latter operator preserves the form $N^{-k} H$, this proves the claim for $G_{\nu}(q)$ and hence all $f$ by continuity.
\end{proof}

If $f: U-\{a\} \rar \HH$ is regular, and meromorphic at $a$, we define the \textit{order} of the pole of $f$ at $a$ as the smallest integer $k>0$ such that $s_{a}^\times(f)N^k \in \HH\dbq{Z_1,Z_2,Z_2}$. We say the pole is \textit{simple} if $k=1$. Equivalently, $f$ has a simple pole at $a$ if $f(q) = G(q-a)b_0 + f_0(q)$, where $b_0\neq 0$, and $f_0(q)$ is regular at $a$.

\section{Elliptic Functions associated with Lattices}

Let $L\subset \HH$ be a full-rank lattice. By a quaternionic elliptic function for $L$ we mean a smooth function $f: \HH-L \rar \HH$ that is quadruply periodic:
$$ f(q+ \lambda) = f(q),\ \ \lambda\in L,\ \ q\in \HH -L.$$

If $f$ is sufficiently well-behaved, for instance left-regular, it has a series expansion whose coefficients are invariants of the lattice $L$. As functions on spaces of lattices, they are often automorphic forms. In this section we construct elliptic functions and the corresponding automorphic forms, and describe their basic properties.

The first steps in this topic were taken by R. Fueter \cite{Fue45} in the 1930s and 40s. It has been generalized since to the context of Clifford analysis by others, notably by R. Krausshar (and collaborators), whose comprehensive Habilitationsschrift \cite{RK2004} also contains valuable surveys and a wealth of references. In this work we shall approach the topic from a more geometric rather than analytic direction, though we still need to provide the basic analytic constructions, which we do in this section. The author discovered quaternionic elliptic functions independently after reading Sudbery's primer on quaternionic analysis \cite{Sud79}, and wishes to apologize for any due credit still left unpaid.

\subsection{Quaternionic Lattices Up to Homothety}

Let $\Lat$ denote the set of all full-rank lattices $L\subset \HH$. We fix an isomorphism $\HH \rar \RR^4, q\mapsto [q]$, where  
$$[x_0 + e_1 x_1 + e_2 x_2 + e_3 x_3] = (x_0,x_1,x_2,x_3).$$
Then $g\in \GL_4(\RR)$ acts on $\HH$ on the right by $q^g = {}^t\rho\cdot (g^{-1}[q])$, $\rho = {}^t(1,e_1,e_2,e_3)$, hence also on $\Lat$. Let $L_0 = \Span_\ZZ(\rho)$, and put $L\cdot g  = \{ q^g: q\in L\}$. Then $g \mapsto L_0\cdot g$ is a surjective map $\GL_4(\RR) \rar \Lat$ that induces a bijection
$$ \GL_4(\ZZ) \backslash \GL_4(\RR) \rar \Lat.$$

The group $\HH^1 \times \HH^1$ also acts on $\HH$ by $(a,b)\cdot q = aqb^{-1}$, and hence also on $\Lat$. We say $L_1$ is \textit{homothetic} to $L_2$ if $L_1=a L_2 b^{-1}$ for some $a,b\in \HH^1$. Let $\Lat^*$ denote the set of homothety classes in $\Lat$. The action of $\HH^1\times \HH^1$ modulo the kernel $\{(1,1),(-1,-1)\}$ factors through $\GL_4(\RR)$ via a map $(a,b) \mapsto g_{a,b}$, characterized by $a Lb^{-1} = L\cdot g_{a,b}$. The image $K$ of $(a,b) \mapsto g_{a,b}$ is maximal compact in $\GL_4(\RR)$, being isomorphic to $\SO(4)$. Thus we have a bijection
$$ \GL_4(\ZZ) \backslash \GL_4(\RR)/K \rar \Lat^*,$$
that maps $\GL_4(\ZZ)g K$ to the homothety class of $L_0\cdot g$, identifying $\Lat^*$ with the symmetric space of $\GL_4(\RR)$.

Any representation $(\rho,V)$ of $\HH^1\times \HH^1$ determines a vector bundle $\calV = \Lat \times_{\HH^1\times\HH^1} V \rar \Lat^*$ the usual way. A section of $\calV$ is then a function $f: \Lat \rar V$ satisfying
$$ f(a L b^{-1}) = \rho(a,b) f(L),\ \ (a,b\in \HH^1).$$

\subsection{Eisenstein Series}

Let $L\subset \HH$ denote a lattice of full rank. For $\nu\in\NN^3$, the quaternionic Eisenstein series is defined by
\begin{align}\label{Enu} E_\nu(L) = \sum_{a\in L^*} G_\nu(a).\end{align}
By (\ref{Gest}), it converges absolutely if $|\nu|>1$, and so provides a function $E_{\nu}: \Lat \rar \HH$. Since $G_{\nu}(q)$ is an odd function if $|\nu|$ is even, $E_{\nu}=0$ in that case. The first non-trivial examples are therefore $E_{210}(L)$, $E_{111}(L)$, \textit{etc.}.

For $n>0$ odd, let $(\rho_n,U_{-n-3})$ denote the representation of $\HH^\times \times \HH^\times$ on the space $U_n$ of regular polynomials of degree $-n-3$, given by
$$ (\rho_n(a,b) f) (q) = bf(a^{-1}qb).$$
Since $\{ G_{\nu}(q): \nu\in \sigma_{n}\}$ is a basis for $U_{-n-3}$, for each pair $a,b\in \HH^1$ there exist constants
$c_{\mu\nu}  = c_{\mu \nu}(a,b)\in \HH$ such that
$$ b G_{\nu}(a^{-1}qb) = \sum_{\mu\in \sigma_n} G_{\mu}(q) c_{\mu \nu}.$$ 
Then
$$bE_\nu(a^{-1}Lb) = \sum_{\lambda\in L^*} bG_{\nu}(a^{-1}\lambda b) = \sum_{\lambda\in L^*} (\rho(a,b) G_{\nu})(\lambda)=\sum_{\lambda\in L^*} \sum_{\mu\in \sigma_n} G_{\mu}(\lambda) c_{\mu\nu} = \sum_{\mu\in \sigma_n} E_{\mu}(\lambda) c_{\mu \nu}.$$
More generally, for any $f\in U_{-n-3}$ we can define
$$  E_n(L,f) = \sum_{\nu\in L^*} f(\lambda),$$
which then satisfies
$$  E_n(a^{-1} L b, f) = b^{-1} E_n(L,\rho_n(a,b)f).$$
Writing $U_{n}^*$ for the space of right $\HH$-linear functions $\alpha: U_{-n-3}\rar \HH$, we can define an action $\rho_n^*$ of $\HH^1\times \HH^1$ on $U_n^*$ by
$$ (\rho_n^*(a,b) \alpha)(f) = b^{-1} \alpha(\rho_n(a,b)f).$$
Then $(\rho_n^*,U_n^*)$ determines a vector bundle $\calU_n$ on $\Lat$, and the equivariant map
$$ \calE_{n}: \Lat \rar U_{n}^*,\ \ \ \calE_n(L)(f) = E_n(L,f),$$
determines a section of the corresponding automorphic vector bundle $\calU_n \rar \Lat^*$. Note that $\calE_n$ is \textit{not} an automorphic form on $\GL_4(\RR)$, but a double cover of it, since the representation of $\HH^1\times \HH^1$ on $U_n^*$ does not factor through $\SO(4)$.

\subsection{Weierstrass $\zeta$-function and $\wp$-function}

For $\nu>0$, by analogy with the complex case the Weierstrass $\wp_\nu$-function is defined as
\begin{align} \wp_\nu(q) = G_\nu(q)  + \sum_{\lambda\in L^*} 
G_{\nu}(q+\lambda)-G_\nu(\lambda).\end{align}
It is a left and right-regular elliptic function for the lattice $L$, with poles on $L$. If $|\nu|\geq 2$, it is evidently equal to
$$ -E_{\nu}(L)+\sum_{\lambda\in L} G_{\nu}(q+\lambda).$$
Since $\partial_{\mu} G_{\nu} = G_{\nu+\mu}$, for $\nu,\mu\neq 0$, the above implies
\begin{align} \label{parmu} \partial_{\mu} \wp_\nu (q) = \wp_{\mu+\nu}(q) + E_{\nu+\mu}(L).
\end{align}

Let $r_L$ be the length of the smallest non-zero element of $L$. For $|q|<r_L$, 
by Proposition \ref{Gexp} we have
$$ G_\nu(q+a) - G_{\nu}(a)= \sum_{\mu >0} (-1)^m 
P_{\mu}(q) G_{\mu+\nu}(a) = \sum_{\mu > 0} (-1)^m 
G_{\mu+\nu}(a)P_{\mu}(q),$$
therefore 
$$ \wp_{\nu}(q) = G_{\nu}(q) + \sum_{a\in L^*} \sum_{\mu>0} (-1)^m P_{\mu}(q) 
G_{\mu+\nu}(a) = G_{\nu}(q)+\sum_{a\in L^*} \sum_{\mu>0} (-1)^m G_{\mu+\nu}(a) 
P_{\mu}(q).$$

Now bringing the outer sum inside, which is justified, we have
\begin{align}\label{wpseries} \wp_{\nu}(q) = G_{\nu}(q) + \sum_{\mu>0} (-1)^m P_{\mu}(q) 
E_{\mu+\nu}(L) = G_{\nu}(q) + \sum_{\mu>0} (-1)^m E_{\mu+\nu}(L) 
P_{\mu}(q).\end{align}
Thus the Eisenstein series appear as series coefficients of the Weierstrass $\wp_{\nu}$-function, as in the complex case. 

The Weierstrass $\zeta$-function is defined by 
\begin{align}\label{zeta}
\zeta(q) = G(q) + \sum_{\lambda\in L^*} \left(G(q+\lambda) - G(\lambda) + \sum_{j=1}^3 G_{(j)}(\lambda)P_{(j)}(q) \right).
\end{align}
It is also left and right regular, but only quasi-periodic, meaning 
$$\zeta( q+ \lambda_i) = \zeta(q) + \eta_i$$ 
for some constants $\eta_i\in \HH$. 

Let $w\in L$ be fixed. We have
	$$ \zeta(q+w)-\zeta(q) = G(q+w) - G(q) + \sum_{\lambda\in L^*} \left\{ G(q+w+\lambda)-G(q+\lambda) + \sum_{j=1}^3 G_{(j)}(\lambda) (P_{(j)}(q+w) - P_{(j)}(q))\right\}.$$
Since $P_{(j)}(q)$ are additive functions, we can rewrite this as
$$ G(q+w) - G(q) + \left(\sum_{\lambda\in L^*} G(q+w+\lambda)-G(q+\lambda) + \sum_{j=1}^3  P_{(j)}(w)G_{(j)}(\lambda) \right).$$
This sum has the same form as the divergent series
$$  \sum_{j=1}^3 P_{(j)}(w)\left(\sum_{\lambda\in L^*} G_{(j)}(\lambda)\right).$$

It therefore suggests an identity of the form 
\begin{align}\eta_i  = \kappa(\lambda_i) + z_1(\lambda_i) E_{100}^*(L)  + z_2(\lambda_i) E_{010}^*(L)  + z_3(\lambda_i) E_{001}^*(L) ,\ \ \ i=1,\cdots,4.\end{align}
where $E_{(j)}^*(L)$ are regularizations of the divergent series $E_{(j)}(L)$, and $\kappa(\lambda_i)$ is a correction factor. In that case, writing $z_j(\lambda_i) = \lambda_i^0 e_i - \lambda_i^j$, we would obtain
\begin{align}\label{etai}
\eta_i = \kappa(\lambda_i) + \lambda_i^0(e_1(E_{100}^*(L)+e_2 E_{010}^*(L)+e_3 E_{001}^*(L))   -  \lambda_i^1 E_{100}^*(L)  -  \lambda_i^2 E_{010}^*(L)  -   \lambda_i^3 E_{001}^*(L) ) .
\end{align}

We will see as a consequence of our main theorem that $\eta_i$ does indeed have such a form. 

\subsection{Projective Embeddings of Quaternionic Tori}

Let $X=\HH/L$ for a lattice $L\subset \HH$. By the quaternionic projective space $\PP^n(\HH)$ we mean the quotient of $\HH^{n+1}-\{0\}$ by the action of $\HH^\times$ given by
$$ a\cdot (x_0,\cdots, x_n) = (x_0a^{-1},\cdots, x_na^{-1} ).$$
We write $[x_0,\cdots, x_n]$ for the class of $(x_0,\cdots, x_n)$, as usual.

For $r>0$ sufficiently small, let $X_r$ denote the image of $\HH-B_0(r)$ in $X$. For any finite set $S=\{\nu_0,\cdots, \nu_n\}\subset \NN^3$ we have a map
$$ I_{S,r}: X_r \rar \PP^s(\HH),\ \ I_{S,r}(q) = [\wp_{\nu_0}(q),\cdots, \wp_{\nu_n}(q)].$$

Let $q_1,q_2\in \HH-B_0(r)$ be distinct points. Then $I_{S,r}(x)=I_{S,r}(y)$ if and only if there exist $a\in \HH^\times$ such that
$$ \wp_{\nu}(q_1)a^{-1} = \wp_{\nu}(q_2),\ \ \ \nu\in S.$$

\begin{lemma}If $q_1,q_2\not\in L$, and $\wp_{\nu}(q_1) = \wp_{\nu}(q_1)$ for all $\nu\in \NN^3$, then $q_1-q_2\in L$.
\end{lemma}
\begin{proof}Assume the hypothesis holds, and let $0\neq \nu\in \NN^3$ be fixed. Then by (\ref{parmu}) for all $\mu>0$, 
$$\partial_{\mu} \wp_{\nu}(q_1) = \wp_{\nu+\mu}(q_1) + E_{\nu+\mu}(L) = \partial_{\mu} \wp_{\nu}(q_2).$$
Then the series expansions (\ref{fexp}) of the regular functions $\wp_{\nu}(q)$ around $q_1$ and $q_2$ coincide, hence $\wp_{\nu}(q-q_1) = \wp_{\nu}(q-q_2)$ for all $q$. Therefore $q_1-q_2$ is a period for $\wp_{\nu}$, and hence it belongs to $L$.
\end{proof}

\begin{prop}If $q_1,q_2,q_1-q_2$ are not in $L$, then there exists a finite set $S\subset \NN^3$ such that $I_{r,S}(q_1)\neq I_{r,S}(q_2).$
\end{prop}
\begin{proof}
Assume $q_1$ and $q_2$ are as in the statement, and that to the contrary $I_{S,r}(q_1)=I_{S,r}(q_2)$ for \textit{all} choices of $S$. For each $k>0$, let $S_k = \{ \nu\in \NN^3: |\nu|\leq k\} = \cup_{n=0}^k \sigma_n$. Then there exist a sequence $a_k\in \HH^\times$ such that
$$  \wp_{\nu}(q_1) a_k^{-1} = \wp_{\nu}(q_2),\ \ \ \nu\in S_k.$$ 
Since $q_1 - q_2\not\in L$, the lemma implies that $\wp_{\nu}(q_1)\neq \wp_{\nu}(q_2)$ for some $\nu$. Then the above implies that $\wp_{\nu}(q_1),\wp_{\nu}(q_2)\neq 0$, and hence the sequence $a_k$ must in fact be constant. Let $a=a_k\in \HH^\times$ so that
$$ \wp_{\nu}(q_1) a^{-1}= \wp_{\nu}(q_2)$$
for \textit{all} $\nu$. 
Now by (\ref{parmu}) and (\ref{fexp}) for $|q|$ small we have
$$ \wp_{\nu}(q+q_i) = \sum_{\mu \geq 0} P_{\mu}(q) (\wp_{\nu+\mu}(q) + E_{\nu+\mu}(L)) (-1)^{m},$$
from which it follows that
$$ \wp_{\nu}(q+q_1)a^{-1} - \wp_{\nu}(q+q_2) = \left(\sum_{\mu\geq 0} P_{\mu}(q) E_{\nu+\mu}(L) (-1)^{m}\right)(a^{-1}-1) = (\wp_{\nu}(q)-G_{\nu}(q)) (a^{-1}-1).$$
The left-hand side is doubly periodic, but $\wp_{\nu}(q) - G_{\nu}(q)$ is not. Therefore $a=1$ necessarily, and $\wp_{\nu}(q_1) = \wp_{\nu}(q_2)$ for all $\nu$. The lemma then implies $q_1 - q_2 \in L$, which contradicts the hypothesis.
\end{proof}

For a fixed $q_1\in X_r$ there is a dense subset of points $q_2\in X_r$ such that no multiple of $q_1 - q_2$ is in $L$, and for each such point $q_2$ we have $I_{S_k,r}(q_1) \neq I_{S_k,r}(q_2)$ for $k$ large enough. In that case there are neighbourhoods $U_1$ and $U_2$, of $q_1$ and $q_2$, respectively, such that  $I_{S_k,r}(q)$ distinguishes every point of $U_1$ from every point of $U_2$. By a standard double compactness argument there is $k$ large enough such that $I_{S_k,r}(q)$ distinguishes every point of $X_r$ from every other. We have therefore proved:

\begin{prop}For every $r$, with $0<r<r_L$, there exists $k$ large enough such that $I_{S_k,r}: X_r \rar \PP^n$ is injective (with $n=|S_k|-1$).
\end{prop}

The proposition is offered only as evidence that the functions $\wp_{\nu}(q)$ separate points on $X$ in a strong sense. We don't expect the image of $I_{S,k}$ to be a ``quaternionic projective variety'' in any naive sense, as there are fundamental obstacles: For one, regular quaternionic functions $f: B_r(0)^\times \rar \HH$ with a pole at $0$ can not be extended to $B_r(0) \rar \PP^1(\HH)$ by setting ``$f(0)=\infty$'' since $\lim_{q\rar 0} f(q)$ is usually undefined. Furthermore, the zero-sets of quaternionic polynomials in a single variable only have the right cardinality up to conjugation, so any notion of a quaternionic algebraic curve must incorporate that ambiguity in higher dimension.

\section{Period Relations}

Our main result will relate periods of differential forms on $\HH/L$ in degrees one and three. For this we will make use of an explicit basis of the homology groups $H_i(\HH/L,\ZZ)$.

\subsection{Main Theorem}
	
Let $\lambda_1,\cdots, \lambda_4$ be a $\ZZ$-basis of the lattice $L\subset \HH$, and put $X=\HH/L$.

 For $1\leq h \leq 4$, we define
$$ \gamma_h:[0,1] \rar \HH,\ \ \ \gamma_h(t) =(t-\frac{1}{2}) \lambda_h.$$
Then $\gamma_h$ determine a $\ZZ$-basis $\{[\gamma_h]\}$ of $H_1(X,\ZZ)$. More generally, for a non-empty subset $\{i_1,\cdots,i_k\}$ of $\{1,2,3,4\}$, with $i_1 < \cdots < i_k$, we define
$$ \gamma_{i_1,\cdots, i_k}:[0,1]^k \rar \HH,\ \ \ \gamma_{i_1,\cdots, i_k}(t_1,\cdots, t_k) =\gamma_{i_{1}}(t_{1})+\cdots + \gamma_{i_k}(t_k).$$

We write $\xi$ for $\gamma_{1,2,3,4}$, whose class is a generator $[X]$ of $H_4(X,\ZZ)$.

If $\gamma: [0,1]^n \rar \HH$ is an $n$-cube, and $1 \leq r \leq n$, for $\epsilon=0,1$ we have the face cubes
$$ \gamma^{r,\epsilon}: [0,1]^{n-1} \rar \HH,\ \  \gamma^{r,0}(t_1,\cdots, t_{n-1}) = \gamma(t_1,\cdots, t_{r-1},\epsilon,t_r,\cdots, t_{n-1}).$$
Then the boundary of $\gamma$ is the $(n-1)$-chain
$$ \partial \gamma = \sum_{i=1}^n \sum_{\epsilon=0,1} (-1)^{i+\epsilon} \gamma^{i,\epsilon}.$$

We write $\xi^{(i)}$ for $\xi^{i,0} = \gamma_{1,2,3,4}^{1,0}$. Setting $\ell_i(q) = q + \lambda_i$, we have
\begin{align*} \partial \xi &= \sum_{i=1}^4 \sum_{\epsilon=0,1} (-1)^{(i+\epsilon)} \xi^{i,\epsilon} = (\xi^{1,1}-\xi^{1,0}) - (\xi^{2,1}-\xi^{2,0}) + (\xi^{3,1}-\xi^{3,0}) - (\xi^{4,1} - \xi^{4,0})\\
&= (\ell_1 \circ  \xi^{1,0} - \xi^{1,0}) - (\ell_2 \circ \xi^{2,0} - \xi^{2,0}) + (\ell_3 \circ \xi^{3,0} - \xi^{3,0}) - (\ell_4 \circ \xi^{4,0} - \xi^{4,0}).
\end{align*}
Then if $\eta$ is a $3$-form on $X$, 
\begin{align}\label{prelim}\int_{\partial \xi} \eta = \sum_{i=1}^4 (-1)^{i+1} \int_{\xi^{(i)}} \ell_i^* \eta - \eta.\end{align}

We write the lattice generators as
$$\lambda_i = \sum_{j=0}^4 \lambda_i^j e_j, \ \ \ \lambda_i^j\in \RR,$$
and we let
$$ \Lambda =\left(\begin{array}{cccc} \lambda_1^0 & \lambda_1^1 & \lambda_1^2 & \lambda_1^3\\
\lambda_2^0 & \lambda_2^1 & \lambda_2^2 & \lambda_2^3\\
\lambda_3^0 & \lambda_3^1 & \lambda_3^2 & \lambda_3^3\\
\lambda_4^0 & \lambda_4^1 & \lambda_4^2 & \lambda_4^3\end{array}\right).$$
For $0\leq \alpha\leq 3$ and $1\leq h \leq 4$, we write $\Lambda_{\alpha h}$ for the matrix obtained by the removing the row and column containing $\lambda_{h}^\alpha$. We put
$$ \wh{\Lambda}_{h\alpha} = (-1)^{\alpha + h +1 }\det \Lambda_{\alpha h},$$
so that the matrix $\wh{\Lambda} = (\wh{\Lambda}_{\alpha h}; 0\leq \alpha \leq 3, 1\leq h \leq 4)$ 
is given by
$$ \wh{\Lambda}  = \det(\Lambda) \Lambda^{-1}.$$

For $1\leq i < j \leq 4$ and $0\leq \alpha < \beta \leq 3$, let
$$\Lambda_{ij,\alpha \beta} = \left(\begin{array}{cc}\lambda_{i}^\alpha& \lambda_{i}^\beta\\ \lambda_j^{\alpha} & \lambda_j^{\beta}\end{array}\right).$$
Let $\Lambda^{(2)}$ denote the $6\times 6$ matrix with rows indexed by pairs $(i,j)$, $1\leq i < j\leq 4$ and columns indexed by $(\alpha,\beta)$, $0\leq \alpha< \beta \leq 3$, both in lexicographic order, and the entry at $(i,j)$ and $(\alpha,\beta)$ given by
$$ \det \Lambda_{ij,\alpha \beta} = \lambda_i^\alpha \lambda_j^\beta - \alpha_i^\beta \lambda_j^\alpha.$$

\begin{lemma}\begin{itemize}
		\item[(i)] Let $1 \leq h \leq 4$ and $0 \leq \alpha_1 < \alpha_2 < \alpha_3 \leq 3$. Then
	$$ \int_{\xi^{(h)}} dx_{\alpha_1} \wedge dx_{\alpha_2} \wedge dx_{\alpha_3} = \det \Lambda_{\alpha h},$$
	where $\alpha$ is the complement of $\{\alpha_1,\alpha_2,\alpha_3\}$ in $\{0,1,2,3\}$.
	\item[(ii)]
	For $1 \leq i < j \leq 4$, and $0\leq \alpha < \beta \leq 3$, we have
	$$ \int_{\gamma_{i,j}} dx_{\alpha} \wedge dx_{\beta} = \det \Lambda_{ij,\alpha\beta}.$$
	\item[(iii)]
	For $1 \leq i \leq 4$ and $0\leq \alpha \leq 3$,
	$$ \int_{\gamma_i} dz_j = \lambda_i^0 e_j - \lambda_i^j.$$
\end{itemize}
\end{lemma}

\begin{proof}Let $i,j,k$ denote the complement of $h$ in $\{1,\cdots,4\}$, with $i<j<k$. Then
	$$ \xi^{(h)} (t_1,t_2,t_3) = \gamma_i(t_1) + \gamma_j(t_2) + \gamma_k(t_3)$$
	and
	\begin{align*}(\xi^{(h)})^* dx_{\alpha_1} = \lambda_i^{\alpha_1} dt_1 + \lambda_j^{\alpha_1} dt_2 + \lambda_k^{\alpha_1} dt_3,\\
	(\xi^{(h)})^* dx_{\alpha_2} = \lambda_i^{\alpha_2} dt_1 + \lambda_j^{\alpha_2} dt_2 + \lambda_k^{\alpha_2} dt_3,\\
	(\xi^{(h)})^* dx_{\alpha_3} = \lambda_i^{\alpha_3} dt_1 + \lambda_j^{\alpha_3} dt_2 + \lambda_k^{\alpha_3} dt_3,\end{align*}
	from which it follows that
	$$ (\xi^{(h)})^*( dx_{\alpha_1} \wedge dx_{\alpha_2} \wedge dx_{\alpha_3}) = (\det \Lambda_{\alpha h}) dt_1 dt_2 dt_3.$$
	Part $(i)$ of the lemma follows by integrating over $[0,1]^3$. Part $(ii)$ is similar and part $(iii)$ is easy.
\end{proof}

The following corollary is then immediate from the definition (\ref{Dq}) of $Dq$ and the fact
that 
 $$dq\wedge dq = e_1 dx_2 \wedge dx_3 + e_2 dx_3 \wedge dx_1 + e_3 dx_1 \wedge dx_2.$$

\begin{cor}\begin{itemize}\item[(i)] For $1\leq h \leq 4$, 
		\begin{align*}\int_{\xi^{(h)}} Dq = \det \Lambda_{0h} - e_1 \det \Lambda_{1h} + e_2 \det \Lambda_{2h} - e_3 \det \Lambda_{3h} = (-1)^{h+1} \sum_{\alpha=0}^3 e_\alpha \wh{\Lambda}_{h\alpha}. \end{align*}
		\item[(ii)] For $1 \leq i < j \leq 4$,
		\begin{align*} \int_{\gamma_{i,j}} dq \wedge dq= e_1 \det \Lambda_{ij,^{23}} - e_2 \det \Lambda_{ij,13} + e_3 \det \Lambda_{ij,12}.  & & & & & & \end{align*}
	\end{itemize}
\end{cor}

For $d=0,\cdots,4$, let $\calE^{(d)}=\calE^{(d)}(X)$ denote the right $\HH$-vector space of closed $\HH$-valued differential $d$-forms $\omega$ on $\HH-L$, that are $L$-periodic, meaning $\ell_{i}^* \omega = \omega$ for $i=1,\cdots,4$. 

For $\omega_1\in\calE^{(1)}$, we write 
$$P(\omega_1) = {}^t (\int_{\gamma_1} \omega_1,\cdots, \int_{\gamma_4} \omega_1).$$
Each such form can be written as $\omega_1 = df$, for $f$ a differentiable $\HH$-valued function on $\HH-L$. If $P(\omega_1) = {}^t (\eta_1,\cdots, \eta_4)$, then $f$ satisfies the quasi-periodicity condition
$$ f(q + \lambda_i) = f(q) + \eta_i,\ \ \ i=1,\cdots,4,$$
for all $q\in \HH-L$. 

Similarly, for $\omega_3\in \calE^{(3)}(X)$, we write
$$ P(\omega_3) = {}^t (\int_{\xi^{(1)}} \omega_3,\cdots,\int_{\xi^{(4)}} \omega_3).$$

Let $J=\diag(1,-1,1,-1)$. 

For a differential $3$-form $\omega$ defined on a domain containing $\partial B_r(0)$, we put
$$ C_r(\omega) = \int_{\partial B_r(0)} \omega.$$

\begin{prop}\label{fprop}Let $\omega_1=df \in \calE^{(1)}$, $\omega_3\in \calE^{(3)}$, and $U_r = \mathrm{int}(\xi) - B_r(0)$ for $r>0$ sufficiently small. Then
		$$ \int_{U_r} \omega_3 \wedge \omega_1 = C_r (\omega_3 f) - {}^t P(\omega_3) J P(\omega_1).$$
\end{prop}
\begin{proof}
Since  $\ell_h^*(\omega_3 f) - \omega_3 f = \omega_3 \eta_h$, by (\ref{prelim}) we have
$$ \int_{\partial \xi} \omega_3 f = \sum_{h=1}^4 (-1)^{h+1} \int_{\xi^{(h)}} \omega_3 \eta_h =\sum_{h=1}^4 (-1)^{h+1} \left( \int_{\xi^{(h)}} \omega_3 \right) \cdot \left( \int_{\gamma_{h}} \omega_1 \right) = {}^t P(\omega_3) J P(\omega_1).$$
Then by Stokes's theorem
$$ \int_{U_r} \omega_3 \wedge \omega_1 = -\int_{\partial \xi} \omega_3 f + \int_{\partial B_r(0)} \omega_3 f = C_r(\omega_3 f) - {}^t P(\omega_3) J P(\omega_1) .$$
\end{proof}

We now consider maps 
$$ \calR,\ \calL: \calE^{(1)}(X) \rar \calE^{(3)}(X),\ \ \ \calL(\omega)= dq \wedge dq \wedge \omega,\ \ \calR(\omega) = \omega \wedge dq \wedge dq.$$
If $\omega = df$, then
$$ \calL(\omega) = d(dq \wedge dq f),\ \ \ \calR(\omega) =d (f dq \wedge dq).$$

For $1\leq i< j \leq 4$, we let
$$ q_{ij} = \int_{\gamma_{i,j}} dq \wedge dq = e_1 \det \Lambda_{ij,23} - e_2 \det \Lambda_{ij,13} + e_3 \det \Lambda_{ij,12},$$
and put
\begin{align}\label{Q} Q = \left(\begin{array}{cccc} 0 & q_{34} & -q_{24} & q_{23}\\ -q_{34} & 0 & q_{14} & -q_{13} \\ q_{24} & -q_{14} & 0 & q_{12} \\ -q_{23} & q_{13} & -q_{12} & 0 \end{array}\right).\end{align}

\begin{lemma}\label{Rlem}$P(\calL(\omega)) = JQ P(\omega),\ \ \ P(\calR(\omega)) = P(\omega) JQ$
\end{lemma}
\begin{proof}Let $\omega = df$, and $\eta = P(\omega)$. The lemma follows from the calculations
\begin{align*} \int_{\gamma_{i,j,k}} \calL(\omega) = \int_{\partial \gamma_{i,j,k}} dq \wedge dq f = \left(\int_{\gamma_{j,k}} dq \wedge dq\right) \eta_i - 
\left(\int_{\gamma_{i,k}} dq \wedge dq\right) \eta_j +
\left(\int_{\gamma_{i,j}} dq \wedge dq\right) \eta_k,
\end{align*}
and 
\begin{align*} \int_{\gamma_{i,j,k}} \calR(\omega) = \int_{\partial \gamma_{i,j,k}} dq \wedge dq f = \eta_i \left(\int_{\gamma_{j,k}} dq \wedge dq\right) - 
\eta_j \left(\int_{\gamma_{i,k}} dq \wedge dq\right) +
\eta_k \left(\int_{\gamma_{i,j}} dq \wedge dq\right).
\end{align*}
\end{proof}

We say a closed $1$-form $\omega$ is \textit{(left-) regular} if $Dq \wedge \omega = 0$, and \text{right-regular} if $\omega \wedge Dq = 0$. If $\omega = df$, that's equivalent to $f$ being left or right regular, respectively.

\begin{lemma}\label{dvlem} Suppose that $\omega = df$ is a regular $1$-form. Then
	$$ \calR(\ol{\omega})\wedge \omega = -|\partial_l f|^2 dv,$$
	where $dv = dt \wedge dx_1 \wedge dx_2 \wedge dx_3$	is the standard volume form.
\end{lemma}
\begin{proof}
	Note that $\ol{dq \wedge dq} = - dq \wedge dq$, and $\ol{\omega} = d \ol{f}$, so that
	$$ \ol{\calL(\omega)} = -d( \ol{f} dq \wedge dq ) = -\calR(\ol{\omega}).$$
	Since $f$ is left-regular, it has a left-derivative $\partial_l f$ and
	$$ \calL(\omega) = d(dq \wedge dq f) = Dq (\partial_l f),$$
	so that
	$$ \calR(\ol{\omega}) = - (\ol{\partial_l f}) \ol{Dq}.$$
	On the other hand, since $\ol{\partial}_l f = 0$, 
	$$ df = \left( - e_1 \pard{f}{x_1} - e_2\pard{f}{x_2}-e_3 \pard{f}{x_3}\right) dt +\pard{f}{x_1}dx_1 + \pard{f}{x_2}dx_2 + \pard{f}{x_3}dx_3 = - \left( dz_1 \pard{f}{x_1} + dz_2 \pard{f}{x_2} + dz_3 \pard{f}{x_3}\right),$$
	where $z_j = te_j - x_j$, $j=1,2,3$. Now from
	$$\ol{Dq} \wedge dz_j= 2e_j dv,$$
	it follows that
	$$\ol{Dq} \wedge df = -2 dv \sum_{j=1}^3 e_i \pard{f}{x_i}=  (\partial_l f)dv.$$
	Then
	$$ \calR({\ol{\omega}}) \wedge \omega = -(\ol{\partial_l f}) \ol{Dq} \wedge df = -|\partial_l f|^2 dv.$$
\end{proof}

Let
\begin{align}\label{lhat} \wh{\lambda} = \wh{\Lambda}\rho = \det(\Lambda)\Lambda^{-1}\rho,\ \ \ \rho = {}^t (1,e_1,e_2,e_3).\end{align}

Our main theorem is as follows.

\begin{thm}[Period Relations]\label{thm1}Let $\omega, \omega'\in \calE^{(1)}(X)$, and write $\omega = df$ for a quasi-periodic function $f$.  
\begin{itemize}\item[(i)] If $U_r = \mathrm{int}(\xi) - B_r(0)$ for $r>0$ sufficiently small,
	$${}^t P(\omega') Q  P(\omega) = - \int_{U_r} \calR(\omega') \wedge \omega - C_r(\calR(\omega')f) .$$
\item[(ii)] Assume furthermore that $\omega=df$ and $\omega'=dg$ are meromorphic with poles at $0$.  Then 
\begin{align*} {}^t \ol{P(\omega)} Q P(\omega) &= -\int_{U_r} |\partial_l f|^2 dv - C_r (\calR(\ol{\omega}) f),\\
{}^t P(\omega') Q P(\omega) &= - C_r( \calR(\omega')f),\\
{}^t \wh{\lambda}\cdot P(\omega) &= 2\pi^2 \mathrm{res}_0(f).
\end{align*}
\item[(iii)] Assume that $\omega$ and $\omega'$ are regular on $X$. Then
\begin{align*} {}^t \ol{P(\omega)} Q P(\omega) &= -\int_{X} |\partial_l f|^2 dv, \\
{}^t P(\omega') Q P(\omega) &= 0,\\
{}^t \wh{\lambda}\cdot P(\omega) &= 0.
\end{align*}
\end{itemize}
\end{thm}
\begin{proof} Part $(i)$ follows from taking $\omega_3 = \calR(\omega')$, $\omega_1 = \omega$ in Proposition \ref{fprop}, and using Lemma \ref{Rlem}. 
	
	The first identity of part $(ii)$ follows from part $(i)$ by setting $\omega' = \ol{\omega}$, and using Lemma \ref{dvlem}. The second identity follows from the fact that if $f$ and $g$ are regular,
	$$ \calR(\omega') \wedge \omega = (\partial_r g) Dq \wedge df = 0.$$
	
For the third identity of part $(ii)$, we note that if $f$ is meromorphic, the form $Dq f$ is closed in its domain. Writing $\eta = P(df)$, we have $\ell_{h}^*(Dq f) - Dqf = (Dq)\eta_h$, for $1\leq h \leq 4$. Then
\begin{align*} \frac{1}{2\pi^2} \mathrm{res}_{0}(f) = \int_{\partial \xi} Dq f &= \sum_{h=1}^4 (-1)^{h+1} \int_{\xi^{(h)}} (\ell_h^* Dqf - Dqf) = \sum_{h=1}^4 (-1)^{h+1} \left(\int_{\xi^{(h)}} Dq\right) \eta_h\\
&= \sum_{h=1}^4 \sum_{\alpha=0}^3 e_{\alpha} \wh{\Lambda}_{h\alpha }\eta_h  = {}^t\rho\cdot ({}^t \wh{\Lambda})  \cdot \eta = {}^t \wh{\lambda} \cdot \eta,\end{align*}
where $\rho={}^t(1,e_1,e_2,e_3)$.

For part $(iii)$, since the forms are defined everywhere, proceeding as in part $(i)$ and $(ii)$ we can integrate over $X$ instead of $U_r$, the $C_r$ terms disappear, and $\mathrm{res}_0(f)=0$.
\end{proof}

Let $\calH^{(1)}=\calH^{(1)}(X)$ be the space of closed left-regular $1$-forms defined on all of $X$.

\begin{cor} The period map $P: \calH^{(1)} \rar \HH^4$ is \textit{injective}. 
\end{cor}
\begin{proof}The map $P$ is right $\HH$-linear. It follows from the first identity in part $(iii)$ of the theorm that if $\omega$ is everywhere-regular and non-zero, then $P(\omega)\neq 0$. Therefore $\ker(P)=0$ and $P$ is injective.
\end{proof}

\subsection{Applications to the Weierstrass-$\zeta$}

The Weierstrass function $\zeta(q)$, defined by (\ref{zeta}) is meromorphic and quasi-periodic with simple poles on $L$, with $\mathrm{res}_0(\zeta)=1$. The periods $\eta = {}^t (\eta_1,\cdots, \eta_4) = P(d\zeta)$ are the quasi-periodicity constants. We write
$$  {}^t \Lambda^{-1} \eta = {}^t(E_0,E_1,E_2,E_3).$$

\begin{prop}\label{Eprop}For each $w\in L$, 
	$$ \zeta(q+w)-\zeta(q) = -z_1(w) E_1 - z_2(w) E_2 - z_3 (w) E_3 - \frac{2\pi^2}{\det(\Lambda)} \Re(w).$$
\end{prop}
\begin{proof}
By part $(iii)$ of Theorem \ref{thm1} we have
$${}^t \wh{\lambda}\cdot \eta = 2\pi^2 \so  {}^t \rho {}^t \Lambda^{-1} \eta = \frac{2\pi^2}{\det(\Lambda)} $$
which implies
$$ E_0 = \frac{2\pi^2}{\det(\Lambda)} - e_1 E_1 - e_2 E_2 - e_3 E_3.$$
Plugging back in  $\eta = {}^t \Lambda {}^t (E_0,E_1,E_2,E_3)$ we have
\begin{align*}\eta_i &=  -\lambda_i^0 (-\frac{2\pi^2}{\det(\Lambda)} + e_1 E_1 + e_2 E_2 + e_3 E_3) + \lambda_i^1 E_1 + \lambda_i^2 E_2 + \lambda_i^3 E_3 \\
&= - z_1(\lambda_i) E_1 - z_2(\lambda_i) E_2 - z_3 (\lambda_i) E_3 + \frac{2\pi^2}{\det(\Lambda)} \Re(\lambda_i).
\end{align*}
Then the result holds for $w=\lambda_i$. Since the functions $z_i(w)$ and $\Re(w)$ are each $\ZZ$-linear, it holds in general.
\end{proof}

Comparing the formula obtained for $\eta_i$ with the wished-for expression (\ref{etai}) identifies $E_1$, $E_2$, $E_3$ with regularized values $E_{100}^*(L)$, $E_{010}^*(L)$, $E_{001}^*(L)$ for the divergent Eisenstein series $E_{100}(L)$, $E_{010}(L)$, $E_{001}(L)$.

\textbf{Example:} Let $a,b,c>0$, and put
$$ \lambda = {}^t(1,e_1\sqrt{a}, e_2\sqrt{b},e_3\sqrt{c}).$$
Then $\Lambda = \diag(1, \sqrt{a},\sqrt{b},\sqrt{c})$, and $(\det \Lambda) ({}^t \Lambda^{-1}\Lambda) = \sqrt{abc} I_3$. The resulting identity is
$$ e_1 E_1(L) + e_2 E_2(L) + e_3 E_3(L) = \frac{-\pi^2}{2\sqrt{abc}}.$$

We guess that a careful identification of $E_i$ with an explicit regularization of $E_{(i)}(L)$ based on a fixed summation order, would make this case of the identity equivalent to to $\sum_{n\geq 1} n^{-2}= \frac{\pi^2}{6}$.

\subsection{Quaternionic Multiplication}

In \cite{Fue45} Fueter suggested considering an elliptic function $f(q)$ with respect to a maximal order $\calO$ in a definite quaternion algebra, and then studying $af(aqa)a$, which is elliptic with respect to $a^{-1}\calO a^{-1}$.  Krausshar \cite{RK2004}, in a much more general Clifford analytic context, considered the shifted average of $\zeta(q)$ over representatives of $a^{-1}\calO a^{-1}/\calO$, and obtained a result about the field generated by the sum of Weierstrass $\wp$-functions on such points. We will use the same constructions, slightly modified, for a different purpose.

As in the previous section, let $L\subset \HH$ be a fixed lattice generated by $\lambda={}^t(\lambda_1,\cdots, \lambda_4)$. By $P(\omega)$ we always mean the period vector of a $1$-form with respect to $\lambda$.

Let $L'$ be another lattice containing $L$, $\calF\subset \HH$ a fundamental parallelogram for $\HH/L$, and $\calF'\subset \calF$ one for $\HH/L'$. Let $\calR\subset \calF'$ be a complete set of coset representative for $L'/L$. For a quasi-periodic function $f$ on $\HH-L$, we put
\begin{align}f_{\calR}(q) = \frac{1}{[L':L]}  \sum_{r\in \calR} f(q-r).\end{align}

\begin{lemma}\label{l1}If $f(q)$ is a regular and quasi-periodic function on $\HH-L$, with poles on $L$, then $f_{\calR}$ is quasi-periodic with respect to $L$, has poles only on $L'$ of the same order as $f$ on $L$, and satisfies
	$$ \mathrm{res}_0(f_{\calR}) = \frac{1}{[L':L]} \mathrm{res}_0(f_R),\ \ \ P(d f_{\calR}) = P(df).$$ 
\end{lemma}
\begin{proof}Each $f(q-r)$ has poles on $L+r$ of the same order as $f$ on $L$. Since $L'$ is a disjoint union of $L+r$ for $r\in R$, $f_{\calR}$ has poles exactly on $L'$, of the same order as $f$. Since only the term $f(q-r)$ with $r\in L$ contributes to the residue of $f_{\calR}$ at $0$, we have
	$$ \mathrm{res}_0(f_{\calR}) = \frac{1}{[L':L]} \mathrm{res}_0(f).$$

	Each $f(q-r)$ has the same quasi-periodicity constants as $f$ (with respect to $\lambda$), and there are $[L':L]$ elements in $\calR$, therefore
	$$P(d f_{\calR}) = \frac{1}{[L':L]} \sum_{r\in \calR} P(df) = P(df).$$
\end{proof}

Now let $D$ be a definite division algebra over $\QQ$, $\calO\subset D$ a maximal order in $D$. We fix an isomorphism $D\otimes \RR \simeq \HH$ and consider $D\subset \HH$. Let $L\subset \HH$ be a left-ideal for $\calO$. Then for each $a\in \calO$ there exists an integer matrix $U_a$ such that
$$ a \lambda  = U_a \lambda.$$ 
Then 
\begin{align}\label{Ua}U \mapsto U_a,\ \ \calO \rar M_n(\ZZ)\end{align}
is an additive function satisfying 
$$U_{ab} = U_b U_a,\ \ \det(U_a)=\rN(a)^2.$$

For $a\in \HH^\times$, and $f$ a function on $\HH$, we put
\begin{align}\label{fa}f_a(q) = f(aq)a.
\end{align}
Then if $f$ is left (resp. right)-regular, so is $f_a$.

\begin{lemma}\label{gvec} Let $f$ be a regular quasi-periodic function on $\HH - L$, with poles on $L$. Then $f_{a}(q)$ is again quasi-periodic with respect to $L$. It has poles only on $a^{-1}L$ of the same order as $f$, and satisfies
	$$P(df_a) =U_a P(df) a,\ \ \ \mathrm{res}_0(f_a)  = G(a) \mathrm{res}_0(f) a.$$
\end{lemma}
\begin{proof} 
The function $f_a(q)$ is quasi-periodic with respect to $a^{-1}L$, which contains $L$ by assumption. If $\eta = P(df)$, and $v=(v_1,\cdots, v_4)$ with $v_i\in \ZZ$, we have
$$ f(q+v \cdot \lambda) - f(q) = v\cdot \eta.$$
Since $a (v\cdot \lambda) = (v U_a)\cdot \lambda$, we have 
	$$f_a(q+ v\cdot \lambda) - f_a(q) = (f(aq + (vU_a)\cdot \lambda) - f(aq))a = (vU_a) \cdot \eta a = v\cdot (U_a \eta a),$$
which shows $P(df_a) = U_a P(df) a$.

If $r=\mathrm{res}_0(f)$, then $f(q) = G(q)r + f_0(q)$ near $q=0$, where $\mathrm{res}_0(f_0)=0$. Then
	$$f_a(q) = G(aq)ra + f_0(aq)a = G(q)G(a) ra + f_0(aq)a$$
	near $0$, where again $\mathrm{res}_0(f_0(aq)a)=0$. This shows $\mathrm{res}_0(f_a) = G(a) \mathrm{res}_0(f)a$.
\end{proof}

Now we apply this to the Weierstrass zeta function $\zeta(q)$.

\begin{thm} Let $L\subset \HH$ be a left-ideal for a maximal order $\calO$ in a division algebra $D\subset \HH$. Fix generators $\lambda_1,\cdots, \lambda_4$ for $L$, define $\wh{\lambda}$ as in (\ref{lhat}), and $U_a$ as in (\ref{Ua}). Then for all $a\in \calO$,
	$$ {}^t \wh{\lambda}\cdot U_a \cdot \eta = 2\pi^2 \ol{a}.$$ 
\end{thm}
\begin{proof} The assertion is trivial if $a=0$. Let $a\in \calO^\times$, put $L'=a^{-1}L$ and fix representatives $\calR$ for $L'/L$ as before. By Lemmas \ref{l1} and \ref{gvec}, both $\zeta_a$ and $\zeta_{\calR}$ are quasi-periodic with respect to $L$, and have simple poles only on $a^{-1}L$. Furthermore, since $\mathrm{res}_0(\zeta) = 1$, we have
$$ \mathrm{res}_{0}(\zeta_a) = G(a)a = \frac{1}{\rN(a)}.$$
Since $[L':L]=[L:aL] = \det(U_a) = \rN(a)^2$, we also have
$$ \mathrm{res}_{0}(\zeta_{\calR}) =\frac{1}{\rN(a)^2}.$$
It follows that the function
$$ h= \zeta_R - \frac{1}{\rN(a)}\zeta_a $$
is quasi-periodic with respect to $L$ and regular at $0$. Recalling that $\eta = P(d\zeta)$, we have
$$ P(dh) = \eta - \frac{1}{\rN(a)}(U_a \eta a).$$
Now applying Theorem \ref{thm1} to $dh$, we obtain 
$$ {}^t \wh{\lambda}\cdot (\eta - U_a \eta \ol{a}^{-1}) = 0.$$
On the other hand, by applying it to $d\zeta$ we have
$$ {}^t\wh{\lambda}\cdot \eta = 2\pi^2 \mathrm{res}_0(\zeta) = 2\pi^2,$$

It follows that
\begin{align}\label{rel1} {}^t \wh{\lambda} \cdot (U_a \eta \ol{a}^{-1}) = 2\pi^2.\end{align}
Multiplying on the right by $ \ol{a}$ we obtain the result.
\end{proof}

\bibliographystyle{alpha}
{\small \bibliography{../../headers/refdb}}

\end{document}